\documentclass[a4paper, reqno,12pt]{amsart}
\usepackage{tikz, tikz-3dplot}
\usepackage{amsmath, amsthm, amssymb}

\theoremstyle{plain}
\newtheorem{te}{Theorem}[section]
\newtheorem{lem}[te]{Lemma}

\newtheorem{pr}[te]{Proposition}
\newtheorem{de}[te]{Definition}

\theoremstyle{remark}
\newtheorem{re}[te]{Remark}
\newtheorem*{ack*}{Acknowledgment}

\def\0{{\bf 0}}

\def\T{{\mathbb T}}
\def\R{{\mathbb R}}
\def\E{{\mathbb E}}

\def\C{{\mathbb C}}

\def\P{{\mathbb P}}

\def\nint{\mathop{\diagup\kern-13.0pt\int}}

\def\dist{{\operatorname{dist}\,}}

\def\Ic{{\mathcal I}}\def\Fc{{\mathcal F}}

\def\Tc{{\mathcal T}}\def\Ec{{\mathcal E}}
\def\Bc{{\mathcal B}}
\def\Ac{{\mathcal A}}\def\Qc{{\mathcal Q}}
\def\Pc{{\mathcal P}}

\def\Lc{{\mathcal L}}

\begin{document}

\title[ ]{Fourier decay for curved Frostman measures}

\author{Shival Dasu}
\address{Department of Mathematics, Indiana University, 831 East 3rd St., Bloomington IN 47405}
\email{sdasu@iu.edu}
\author{Ciprian Demeter}
\address{Department of Mathematics, Indiana University, 831 East 3rd St., Bloomington IN 47405}
\email{demeterc@iu.edu}

\thanks{The second author is  partially supported by the NSF grant DMS-2055156}

\begin{abstract}
We investigate Fourier decay for Frostman measures supported on  curves with nonzero curvature. We combine decoupling with known lower bounds for Furstenberg sets to extend the main result in \cite{O}.  	
 	
\end{abstract}

\maketitle

\section{introduction}

We prove the following result.
\begin{te}
\label{t1}
Let $\Gamma$ be the graph of a $C^3$ function $\gamma:[-1,1]\to\R$ satisfying the nonzero curvature condition $\min_{-1\le x\le 1}|\gamma''(x)|>0$.
Let $0<s<1$. Let $\mu$ be a Borel measure supported on $\Gamma$
satisfying the Frostman condition
\begin{equation}
\label{e4}
\mu(B(y,r))\lesssim r^s
\end{equation}
for each $y\in \R^2$ and each $r>0$. Then there exists $\beta>0$ (depending on $s$) such that for each ball $B_R\subset\R^2$ of radius $R$ we have
\begin{equation}
\label{e1}
\|\widehat{\mu}\|_{L^6(B_R)}\lesssim R^{\frac{2-2s-\beta}{6}}.
\end{equation}	
\end{te}
This theorem was proved in \cite{O} in the case $\Gamma$ is the parabola
$$\P^1=\{(\xi,\xi^2):\;|\xi|\le 1\}.$$

The location of $B_R$ plays no role in the argument, so we may and will assume $B_R$ is centered at the origin.
Theorem \ref{t1} has an intrinsic interest stemming from its interplay of curvature and the Fourier transform in the fractal setting. In addition to this, it was observed in \cite{O} that the case $\Gamma=\P^1$ encodes non trivial sum-product information, as it implies the non existence of Borel subrings of the reals with Hausdorff dimension $s\in (0,1)$.
\smallskip

The argument in \cite{O} has two major ingredients. The first one uses in a critical way the fact that $\Gamma=\P^1$.  Following the  discretization of the measure, the $L^6$ statement \eqref{e1} is converted into a 3-energy estimate (see Section \ref{s4} for a definition). In a similar spirit with an elementary  computation from \cite{BD}, this is interpreted as  an incidence problem, between $1/R$-squares and $1/R$-neighborhoods of circles. As the circles are centered on the $x$-axis, the problem can be further reformulated  in the more familiar language of incidences between squares and tubes.

This leads to the second major ingredient in \cite{O}.  More precisely, it is proved that if the 3-energy is too ``large", it forces the existence of a certain ``small"  Furstenberg set (see Section \ref{s6} for a definition). However, lower bounds from the literature on Furstenberg sets show that they have to be ``large", leading to a contradiction.
\smallskip

We provide an alternative argument to Theorem \ref{t1}, that bears both similarities to the one in \cite{O}, but also exhibits key differences and provides new angles that could  potentially make it useful for further advances on the topic. Let us describe how it works.

We also use two main ingredients. The first one is the (regular) $L^6$ decoupling.

\begin{te}
	\label{t2}
	Assume $F:\R^2\to\C$ has the spectrum inside the $1/R$-neighborhood of $\Gamma$. Cover this neighborhood with essentially rectangular boxes $\theta$ with dimensions roughly  $R^{-1/2}$ and $R^{-1}$. Let $F_\theta$ be the Fourier restriction of $F$ to $\theta$. Then
	$$\|F\|_{L^6(\R^2)}\lesssim_\epsilon R^\epsilon(\sum_\theta\|F_\theta\|_{L^6(\R^2)}^2)^{1/2}\;\;\;\;\text{(global version)}$$
	and for each ball $B_R$ of radius $R$
	$$\|F\|_{L^6(B_R)}\lesssim_\epsilon R^\epsilon(\sum_\theta\|F_\theta\|_{L^6(w_{B_R})}^2)^{1/2}\;\;\;\;\text{(local version)},$$
	where $w_{B_R}$ is a smooth approximation of $1_{B_R}$
\end{te}

This tool is efficient for handling the 3-energy of measures supported near $\Gamma$, and it easily implies \eqref{e1} without the gain $R^{-\beta}$ (to be more precise, with arbitrarily small $R^\epsilon$ losses). While decoupling is sharp in many applications, Theorem \ref{t1} is an instance where decoupling needs an ``$\epsilon$-boost"
from incidence geometry, in order to lead to sharper results. We break $\widehat{\mu}$ into wave packets supported on tubes. Then we partition the spatial domain into {\em light }squares -those intersecting only few tubes- and {\em heavy} squares. Refined decoupling, Theorem \ref{t29}, easily takes care of the $L^6$ mass of $\widehat{\mu}$ restricted to the light squares. The heavy squares need a more delicate argument that brings to bear the second main ingredient: an improved incidence estimate for tubes, that we derive from the recent progress on  Furstenberg sets.
\smallskip

A surprising feature of the $L^6$ argument in \cite{O} is that it takes place on the {\em spectral side}, where $\mu$ lives. This is possible due to the special nature of the exponent 6, being an even number, but also due to the special nature of the parabola.  Our argument takes place on the {\em spatial side}. We start by breaking $\widehat{\mu}$ into wave packets. The Furstenberg set we construct consists of the heavy squares, and they live on the spatial side. Another interesting difference is that \cite{O} uses $(\delta,s,2s,C)$-sets, while we use $(\delta,1-s,1,C)$-sets. See Definition \ref{d17}. The fact that the upper-regularity \eqref{e4} of the measure $\mu$ forces the needed $(\delta,1-s,1,C)$-discretized structure is a consequence of the fact that tubes point in the directions determined by arcs $\theta$ on $\Gamma$, with prescribed mass $\mu(\theta)\approx R^{-s/2}$ close to the maximum value allowed by \eqref{e4}.
\smallskip

There are a few advantages of our argument. On the minor side, it does not need the fine discretization at scale $1/R$, present in \cite{O}. The pigeonholing (P1) in Section \ref{s1} is a non technical description of the fact that, essentially, the restriction $\mu_\theta$ of $\mu$  to $\theta$ is concentrated on the union of $M$ arcs of size $1/R$. The additive structure of these arcs, whose analysis occupies a large part of the proof in \cite{O}, is efficiently taken care by our application of decoupling. Also, we do not make any serious use of transversality,  this feature is only apparent in the proof of the very standard estimate \eqref{e31}.

On the more serious side, Theorem \ref{t1} covers all curved $\Gamma$. It is the special  nature of $\P^1$ that allows for the conversion of the  3-energy into the language of incidences between lines/tubes and points/squares. Such a mechanism is not available for arbitrary curves. 
Also, our proof does not depend in any way on the fact that $6$ is an even integer. This is because decoupling also does not rely on this feature. Given  this, we believe that the argument presented here is potentially  applicable for other problems. One such example is the case of Frostman measures supported on other manifolds, for which the correct index replacing $6$ is fractional. These questions may be connected with novel  (e.g. sum-product) phenomena in high dimensions. We must caution however that the analogues of Furstenberg sets in higher dimensions are poorly understood, and likely rather difficult.

This begs another question: can one prove Theorem \ref{t1}, by avoiding the use of Furstenberg sets? At some level, this may seem counterintuitive; the paper \cite{KT} shows the essential  {\em equivalence} between the discretized ring conjecture and the existence of non trivial lower bounds for sets of Furstenberg-type.

However, we can give a partial answer to the question, if in addition to the upper regularity \eqref{e4} we also assume lower regularity  (such a measure is called AD-regular)
\begin{equation}
\label{e4lo}
\mu(B(y,r))\gtrsim r^s, \;\; \text {for each }y\in\Gamma,\; r\le 1.
\end{equation}
In this case, decoupling reduces \eqref{e1} to proving some sensible estimates for $\|\widehat{\mu_\theta}\|_{L^6(B_R)}$.
Such estimates became recently available, see the discussion in Section \ref{s0}, and (at least some of) their proofs seem (at least on the surface) fairly disjoint from the circle of ideas surrounding sum-product theory. This  provides an alternative argument for \eqref{e1} in the AD-regular case. That being said, Frostman measures are typically not lower-regular. An interesting related example is the construction in \cite{EV} of additive subgroups of the reals of arbitrary Hausdorff dimension $s\in (0,1)$. If their Frostman measures were regular, they would be forced to satisfy improved energy estimates, contradicting the  additive structure.

There are many examples showing decoupling encodes sum-product structure, relevant to 3-energies. To point out just one of them, if $A\subset [1,2]$ is $\delta$-separated, then we have
$$|\{(a_1,\ldots,a_6)\in A^6:\;a_1+a_2+a_3=a_4+a_5+a_6\text{ and }a_1a_2a_3=a_4a_5a_6\;\}|\lesssim_\epsilon \delta^{-\epsilon}|A|^3.$$
This is an immediate application of Theorem \ref{t2} to the $\log$ function.

In our argument, we only use decoupling at two spatial scales, $R$ and $R^{1/2}$. One question remains on whether a multi-scale decoupling could exploit lower regularity of $\mu$ at appropriate scales, and allow for the efficient use of the aforementioned improved energy estimates. The heuristic for this is that $\mu$ may be (very informally) thought of as consisting of two parts: the AD-regular part with exponent $s$ and the upper-regular part associated with larger exponents. The latter part is ameanable to induction, as the upper bound in \eqref{e1} becomes more favorable as $s$ increases.

\bigskip

We close the introduction with a word about hidden constants. The application of decoupling will be one of the sources for $R^\epsilon$ losses in the argument. Other similar losses come from pigeonholing, and the fact that wave packets are not fully spatially supported on tubes. All these losses can be made arbitrarily small by choosing $\epsilon$ small enough. They will be compensated by the (fixed) gain guaranteed by the incidence geometry of the Furstenberg sets. We will be somewhat casual in keeping track of the losses. In fact, in order to not sacrifice  the clarity of the exposition, most of the losses will be hidden under the rug.

\begin{ack*}
We are grateful to Tuomas Orponen for providing feedback on the manuscript.	
\end{ack*}

\section{Proof of Theorem \ref{t1}}

\subsection{A warm up: energies of measures and the proof in the AD-regular case}
\label{s0}
Let $\nu$ be a Borel measure on the real line. We define its $2$ and $3$-energies at scale $r>0$ as follows
$$\Ec_2(\nu,r)=\int_{\R^3}\nu([x_1+x_2-x_3-r,x_1+x_2-x_3+r])d\nu(x_1)d\nu(x_2)d\nu(x_3),$$
$$\Ec_3(\nu,r)=\int_{\R^5}\nu([x_1+x_2+x_3-x_4-x_5-r,x_1+x_2+x_3-x_4-x_5+r])d\nu(x_1)\ldots d\nu(x_5).$$
It is clear that if $\nu$ is an upper-regular (i.e satisfies \eqref{e4}) {\em probability} measure, then
$$\Ec_3(\nu,r)\lesssim \Ec_2(\nu,r)\lesssim r^s.$$
These estimates can be sharp. An example is when $\nu$ is supported on $r^{-s}$ intervals $I$ of length $r$ in $[0,1]$, with centers in arithmetic progression. On each $I$ we let $\nu$ be the normalized Lebesgue measure $r^{s-1}dx$. It is easy to compute that $\Ec_3(\nu,r)\sim \Ec_2(\nu,r)\sim r^s.$
 
When $\nu$ is also lower-regular according to \eqref{e4lo}, we have the following improvement. See \cite{DZ} and \cite{CT}.
\begin{te}
\label{tlasst}	
Assume $\nu$ is some AD-regular probability measure on $[-1,1]$.
Then there is $\beta>0$ depending on $s$ such that for each $0<r\le 1$ we have
$$\Ec_3(\nu,r)\lesssim \Ec_2(\nu,r)\lesssim r^{s+\beta}.$$
\end{te}

We can use this to give a quick proof of Theorem \ref{t1} if $\mu$ is AD-regular. This property forces $\mu$ to be supported on roughly $R^{s/2}$ arcs $\theta$ of length $R^{-1/2}$.  First, we apply the local version of decoupling, Theorem \ref{t2}
$$\|\widehat{\mu}\|_{L^6(B_R)}\lesssim_{\epsilon}R^{\epsilon}(\sum_{\theta}\|\widehat{\mu_\theta}\|_{L^6(w_{B_R})}^2)^{1/2}\lesssim R^{\frac{s}{4}+\epsilon}\max_{\theta}\|\widehat{\mu_\theta}\|_{L^6(w_{B_R})}.$$
Fix $\theta$. Let $\nu$ be the measure supported on the real line such that
$$\nu(S)=\mu_\theta(\{(\xi,\gamma(\xi)):\;\xi\in S\}).$$
We have
$$\widehat{\mu_\theta}(x)=\int e((\xi,\gamma(\xi))\cdot x)d\nu(\xi).$$
It follows that
$$\|\widehat{\mu_\theta}\|_{L^6(w_{B_R})}^6=\int_{\R^6}\widehat{w_{B_{R}}}(\xi_1+\ldots-\xi_6,\gamma(\xi_1)+\ldots-\gamma(\xi_6))\;d\nu(\xi_1)\ldots d\nu(\xi_6).$$
We may assume $|\widehat{w_{B_R}}|\lesssim R^21_{B(0,R^{-1})}$. This restricts the domain of integration to $|\xi_1+\xi_2+\xi_3-\xi_4-\xi_5-\xi_6|\le R^{-1}$ and leads to the upper bound
$$\|\widehat{\mu_\theta}\|_{L^6(w_{B_R})}^6\lesssim R^2\Ec_3(\nu,R^{-1}).$$
The rescaled version of Theorem \ref{tlasst} shows that
$$\Ec_3(\nu,R^{-1})\lesssim \mu(\theta)^5(\frac1R)^{s+\beta}\sim R^{-\frac{7s}{2}-\beta}.$$
Putting all these estimates together we arrive at the desired conclusion
$$\|\widehat{\mu}\|^6_{L^6(B_R)}\lesssim_{\epsilon}R^{\epsilon} R^{\frac{3s}{2}}R^2R^{-\frac{7s}{2}-\beta}=R^{2-2s+\epsilon-\beta}.$$

\medskip

\subsection{Pigeonholing and the fractal distribution of wave packets}
\label{s1}
Fix a smooth $\psi:\R^2\to \R$ satisfying $0\le \psi\le 1_{B(0,1)}$. Write $\psi_\delta(\xi)=\delta^{-2}\psi(\frac{\xi}{\delta})$ and $\mu_\delta=\mu*\psi_\delta$. Then \eqref{e1} can  be reformulated equivalently as
\begin{equation}
\label{e2}
\|\widehat{\mu_{1/R}}\|_{L^6(\R^2)}^6\lesssim R^{{2-2s-\beta}}.
\end{equation}
Note that $\mu_{1/R}$ is supported on the $1/R$-neighborhood of $\Gamma$.
We partition this neighborhood into essentially rectangular regions $\theta$ of dimensions roughly $R^{-1/2}$ and $1/R$. We also partition the $1/R^{1/2}$-neighborhood into  sets $\tau$, with dimensions roughly $R^{-1/4}$ and $R^{-1/2}$. Each $\theta$ sits inside some $\tau$.
\smallskip

Let $D,M,P$ be dyadic parameters, with $D$ and $P$  integers. We decompose $\mu$
$$\mu=\sum_{D,M,P}\mu_{D,M,P},$$
with each $\mu_{D,M,P}$ satisfying the following properties:
\\
\\
(P1) for each $\theta$, either $\mu_{D,M,P}(\theta)=\mu(\theta)\sim MR^{-s}$, or  $\mu_{D,M,P}(\theta)=0$. We call $\theta$ {\em active} if it falls into the first category.
\\
\\
(P2) each $\tau$ contains either $\sim P$ or no active $\theta$. We call $\tau$ {\em active} if it fits into the first category.
\\
\\
(P3) there are $\sim D$ active $\theta's$, or equivalently, there are $\sim D/P$ active $\tau's$.

We note that these properties together with \eqref{e4} force the following inequalities
\begin{equation}
\label{e5}
M\lesssim R^{s/2}
\end{equation}
\begin{equation}
\label{e6}
DM\lesssim R^s
\end{equation}
\begin{equation}
\label{e7}
MP\lesssim R^{3s/4}.
\end{equation}
\smallskip
The parameter $M$ will guarantee  good control over the distribution of the active $\theta$'s, which will in turn deliver the much needed discretized structure for our tubes.
\smallskip

For the rest of the argument we fix $D,M,P$ and let $F=R^{s-2}\widehat{\mu_{D,M,P}*\psi_{1/R}}$ be the $L^\infty$ upper normalized version, $\|F\|_\infty\lesssim 1$. Invoking the triangle inequality, it suffices to prove that for some $\beta>0$ we have
\begin{equation}
\label{e3}
\|F\|_{L^6}^6\lesssim R^{4s-10-\beta}.
\end{equation}
This is because only $O((\log R)^{3})$ of the 3-tuples $(D,M,P)$ are relevant to the summation. Indeed, we have $O(\log R)$ possible choices for each of  $D$ and $P$. The contribution from $M=O(R^{-100})$ is negligible, as it becomes transparent from the main argument that follows. Given the upper bound  \eqref{e5}, we conclude that there are only $O(\log R)$ relevant values for $M$.
\smallskip

Write $\mu_\theta$ for the restriction of $\mu$ to an active  $\theta$.
Let $F_\theta=R^{s-2}\widehat{\mu_\theta*\psi_{1/R}}$ be the Fourier restriction of $F$ to $\theta$, so that
$F=\sum_\theta F_\theta$.
We will use the following spatial decomposition for $F_\theta$, see e.g. Chapter 2 in \cite{book}
$$F_\theta=\sum_{T\in\Tc_\theta}\langle F_\theta,W_T\rangle W_T+O(R^{-100}).$$
Each $T$ is a rectangle (referred to as tube) with dimensions $R^{1/2}$ and $R$, with the long side pointing in the direction normal to $\theta$. The term $O(R^{-100})$ is negligible, and can be dismissed. Its role is to ensure that all $T$ sit inside, say,  $[-R,R]^2$. The wave packet $W_T$ is $L^2$ normalized, has spectrum inside (a slight enlargement of) $\theta$, and is essentially concentrated spatially in (a slight enlargement) of $T$. Thus, with $\chi_T$ being a smooth approximation of $1_T$, we have
$$|W_T|\lesssim R^{-3/4}\chi_T.$$

We group the tubes $T$ according to the magnitude of $|\langle F_\theta,W_T\rangle|$. Given the dyadic parameter $\lambda$ we write
$$\Tc_{\lambda,\theta}=\{T\in\Tc_\theta:\;|\langle F_\theta,W_T\rangle|\sim \lambda\}.$$

We record two key properties. The first one shows that the largest possible value $\lambda_{max}$ of $\lambda$ satisfies
\begin{equation}
\label{e15}
\lambda_{max}\sim MR^{-\frac54}.
\end{equation}
\begin{pr}
For each $T\in\Tc_{\theta}$ we have
$$|\langle F_\theta,W_T\rangle|\lesssim MR^{-\frac54}.$$ 	
\end{pr}	
\begin{proof}
We have
$$R^{2-s}|\langle F_\theta,W_T\rangle|=|\langle \mu_\theta*\psi_{1/R},\widehat{W_T\rangle}|\le \|\mu_\theta*\psi_{1/R}\|_1\|\widehat{W_T}\|_{\infty}\lesssim \|\mu_\theta\|_1R^{3/4}\sim MR^{\frac34-s}.$$
\end{proof}

The second property is about the fractal-like distribution of the tubes $T$ inside boxes. This is essentially Lemma 3.3 in \cite{O}. We sketch a proof for the convenience of the reader.

\begin{pr}Let $\Delta\in [R^{1/2},R]$.
For each rectangle $B$ with dimensions $\Delta$ and $R$, and with orientation identical to that of the tubes in $\Tc_\theta$, we have  the estimate
\begin{equation}
\label{e12}
N_B=|\{T\in\Tc_{\lambda,\theta}:\; T\subset B\}|\lesssim (\frac{\Delta}{\sqrt{R}})^{1-s}\frac{MR^{\frac{s-5}{2}}}{\lambda^2}.
\end{equation}
\end{pr}
\begin{proof}
Let $\eta_B$ be a smooth approximation of $1_B$ with compact support inside $B(0,1/\Delta)$. Using $L^2$ orthogonality and ignoring again negligible terms,  we get
\begin{equation}
\label{e10}
N_B\lambda^2\sim \|\sum_{T\in\Tc_{\lambda,\theta}:\; T\subset B}\langle F_\theta,W_T\rangle W_T\|_2^2\lesssim \|F_\theta\eta_B\|_2^2=R^{2s-4}\|\mu_\theta*\psi_{1/R}*\widehat{\eta_B}\|_2^2.
\end{equation}
Using that $|\widehat{\eta_B}|\lesssim |B|1_{B(0,1/\Delta)}=\Delta R1_{B(0,\Delta^{-1})}$ and $|\psi_{1/R}|\lesssim R^21_{B(0,1/R)}$, we may write (since $\Delta^{-1}\ge 1/R$)
$$\psi_{1/R}*|\widehat{\eta_B}|\lesssim \Delta R1_{B(0,2\Delta^{-1})}.$$
Combining this with  \eqref{e4} at scale $r=\Delta^{-1}\le R^{-1/2}$ we find
\begin{equation}
\label{e8}
\|\mu_\theta*\psi_{1/R}*\widehat{\eta_B}\|_{\infty}\lesssim \Delta R\Delta^{-s}.
\end{equation}
On the other hand
\begin{equation}
\label{e9}
\|\mu_\theta*\psi_{1/R}*\widehat{\eta_B}\|_{1}\le \mu(\theta)\|\psi_{1/R}\|_1\|\widehat{\eta_B}\|_{1}\lesssim \mu(\theta)\sim MR^{-s}.
\end{equation}
Interpolating \eqref{e8} and \eqref{e9} leads to
$$\|\mu_\theta*\psi_{1/R}*\widehat{\eta_B}\|_2^2\lesssim M(\Delta R)^{1-s}.$$
The conclusion follows by combining this with \eqref{e10}
$$N_B\lesssim \lambda^{-2}R^{2s-4}(\Delta R)^{1-s}=(\frac{\Delta}{\sqrt{R}})^{1-s}\frac{MR^{\frac{s-5}{2}}}{\lambda^2}.$$
\end{proof}
\medskip

\subsection{Regular decoupling implies no $R^\beta$ gain}

Recall that we need to prove \eqref{e3}. Let us start by emphasizing the relevance of the gain coming from $\beta>0$. To do so, we first prove the estimate without this gain, in fact we will tolerate arbitrarily small losses
\begin{equation}
\label{e13}
\|F\|_{L^6}^6\lesssim_\epsilon R^{4s-10+\epsilon}.
\end{equation}
 All we need from \eqref{e12} to argue this is the estimate when $B$ is the square $[-R,R]^2$. Recall that all $T$ in $\Tc_{\lambda,\theta}$ are subsets of this square, so \eqref{e12} with $\Delta\sim R$ gives
\begin{equation}
\label{e11}
|\cup_{\theta}\Tc_{\lambda,\theta}|\lesssim \frac{MD}{\lambda^2R^2}.
\end{equation}
Recall also the decomposition
$$F=\sum_{\lambda\le \lambda_{max}}\sum_{\theta\text{ active}}\sum_{T\in\Tc_{\lambda,\theta}}\langle F_\theta,W_T\rangle W_T.$$
Then \eqref{e13} will follow once we prove that for each $\lambda\le \lambda_{max}$ we have
\begin{equation}
\label{e16}
\|\sum_{\theta\text{ active}}\sum_{T\in\Tc_{\lambda,\theta}}\langle F_\theta,W_T\rangle W_T\|_6^6\lesssim_\epsilon R^{4s-10+\epsilon}.
\end{equation}
This is indeed enough to imply \eqref{e13}, since the smaller values of $\lambda$ lead to negligible contributions. To prove this estimate, we first use decoupling, Theorem \ref{t2} (and H\"older's inequality combined with property (P3)), then \eqref{e11}, followed by \eqref{e15}
\begin{align*}
\|\sum_{\theta\text{ active}}\sum_{T\in\Tc_{\lambda,\theta}}\langle F_\theta,W_T\rangle W_T\|_6^6&\lesssim_\epsilon R^\epsilon D^2\sum_{\theta\text{ active}}\|\sum_{T\in\Tc_{\lambda,\theta}}\langle F_\theta,W_T\rangle W_T\|_6^6\\&\lesssim_\epsilon R^\epsilon D^2R^{-9/2}\lambda^6\sum_{\theta\text{ active}}\|\sum_{T\in\Tc_{\lambda,\theta}}\chi_T\|_6^6\\&\lesssim_\epsilon R^\epsilon D^2R^{-9/2}\lambda^6R^{3/2}|\cup_{\theta}\Tc_{\lambda,\theta}|\\&\lesssim_\epsilon R^\epsilon\frac{MD^3\lambda^4}{R^{5}}\\&
\lesssim_\epsilon R^\epsilon\frac{M^5D^3}{R^{10}}.
\end{align*}
Finally, invoking \eqref{e5} and \eqref{e6} we have
$$M^5D^3=(MD)^3M^2\lesssim R^{4s}.$$
The combination of the last two inequalities proves the desired conclusion \eqref{e16}.
\medskip

\begin{re}
	\label{r7}
This argument produces the gain $R^\beta$, if at least one of the following three scenarios holds
\\
\\
(S1) we restrict the summation to values of $\lambda$ slightly smaller than $\lambda_{max}$
\\
\\
(S2)  $M$ is slightly smaller than $R^{s/2}$
\\
\\
(S3) $D$ is slightly smaller than $R^{s/2}$.
\smallskip

In the next sections we will deal with the remaining scenario. More precisely, we may assume that  there is some small enough (to be chosen in the proof of Theorem \ref{t5}) $\alpha>0$ such that
\begin{equation}
\label{e25}
\lambda\ge \frac{\lambda_{max}}{R^\alpha}=\frac{M}{R^{\alpha+\frac54}}
\end{equation}

\begin{equation}
\label{e22}
M\ge R^{\frac{s}2-\alpha}
\end{equation}
\begin{equation}
\label{e23}
D\ge R^{\frac{s}2-\alpha}.
\end{equation}
The choice of $\alpha$ will in the end determine the value of $\beta$ in Theorem \ref{t1}.

For the remainder of the paper we fix such $\lambda,M,D$. We record two consequences, for later use
\begin{equation}
\label{e35}
D\lesssim R^{\frac{s}2+\alpha}\;\;\text{(due to \eqref{e6} and \eqref{e22})}
\end{equation}

\begin{equation}
\label{e36}
|\Tc=\cup_{\theta}\Tc_{\lambda,\theta}|\lesssim R^{\frac12+4\alpha}\;\;\text{(due to \eqref{e11}, \eqref{e25}, \eqref{e22} and \eqref{e35}}).
\end{equation}

To make the refined argument work, we will need to also use the parameter $P$.
\end{re}

\subsection{Refined decoupling and the improved estimate for light squares}
We tile the plane with squares $q$ with side length $\sqrt{R}$. We define
$$\Qc_{heavy}=\{q:\;q\text{ intersects at least }R^{\frac{s}{2}-\alpha}\;\;\text{tubes }T\in\Tc\}.$$
Also, $\Qc_{light}$ are the remaining squares. We first deal with this latter collection, using refined decoupling.
\begin{te}[\cite{Gu}]
\label{t29}Assume $F$ has spectrum inside the $1/R$-neighborhood of $\Gamma$, and the wave packet decomposition $F=\sum_{T\in\T}a_TW_T$.
Let $\Qc$ be a collection of pairwise disjoint squares with side length $\sqrt{R}$. Assume (the slight enlargement of) each $q\in\Qc$ intersects at most N tubes in $\T$.
	Then
	$$\|F\|_{L^6(\cup_{q\in\Qc}q)}\lesssim_\epsilon R^\epsilon N^{\frac13}(\sum_\theta\|F_\theta\|_{L^6(\R^2)}^6)^{1/6}.$$
\end{te}
The advantage of this over the  regular decoupling (Theorem \ref{t2}) used in the previous section is that $N^{\frac13}$ replaces the ``H\"older loss" $D^{\frac{1}3}$. While $D$ could potentially be as large as $R^{\frac{s}{2}+\alpha}$ (see \eqref{e35}), localizing to $\Qc_{light}$ allows us to use the critically smaller value $N=R^{\frac{s}{2}-\alpha}$.
\begin{te}
\label{t9}
We have
$$\|\sum_{\theta\text{ active}}\sum_{T\in\Tc_{\lambda,\theta}}\langle F_\theta,W_T\rangle W_T\|_{L^6(\cup_{q\in\Qc_{light}}q)}^6\lesssim_\epsilon R^{\epsilon-2\alpha}{R^{4s-10}}.$$
\end{te}
\begin{proof}
We repeat the computation from the previous section, incorporating the gain from $N$.
\begin{align*}
\|\sum_{\theta\text{ active}}\sum_{T\in\Tc_{\lambda,\theta}}\langle F_\theta,W_T\rangle W_T\|_{L^6(\cup_{q\in\Qc_{light}}q)}^6&\lesssim_\epsilon R^\epsilon R^{s-2\alpha}\sum_{\theta\text{ active}}\|\sum_{T\in\Tc_{\lambda,\theta}}\langle F_\theta,W_T\rangle W_T\|_6^6\\&\lesssim_\epsilon R^\epsilon R^{s-2\alpha-9/2}\lambda^6\sum_{\theta\text{ active}}\|\sum_{T\in\Tc_{\lambda,\theta}}\chi_T\|_6^6\\&\lesssim_\epsilon R^\epsilon R^{s-2\alpha-9/2}\lambda^6R^{3/2}|\cup_{\theta}\Tc_{\lambda,\theta}|\\&\lesssim_\epsilon R^\epsilon R^{s-2\alpha}\frac{MD\lambda^4}{R^{5}}\\&\lesssim_\epsilon R^\epsilon R^{s-2\alpha}\frac{M^4(MD)}{R^{10}}\\&\lesssim_\epsilon R^{\epsilon-2\alpha}{R^{4s-10}}.
\end{align*}
We have used \eqref{e5} and \eqref{e6}. The exponent gain $\beta=2\alpha-\epsilon$ is positive, if $\epsilon$ is chosen small enough.

\end{proof}
\smallskip

In the remainder of the paper we will deal with $\Qc_{heavy}$. The gain in this case will come once we prove that this collection has small cardinality.

\subsection{Discrete energy estimates, easy and hard}
\label{s4}
This section provides some intuition for the following sections.
Consider a collection $\Pc$ of $\delta$-separated points in $\R^n$. We may define the 2 and 3-energies of $\Pc$ at scale $\delta$ via
$$\E_{2,\delta}(\Pc)=|\{(p_1,p_2,p_3,p_4)\in\Pc^4:\;\dist(p_1+p_2,p_3+p_4)\lesssim \delta\}|,$$
$$\E_{3,\delta}(\Pc)=|\{
(p_1,p_2,p_3,p_4,p_5,p_6)\in\Pc^6:\;\dist(p_1+p_2+p_3,p_4+p_5+p_6)\lesssim \delta\}|.$$
We can also use integrals to measure essentially the same quantities
$$\E_{k,\delta}(\Pc)\approx \frac{1}{|B(0,\delta^{-1})|}\int_{B(0,\delta^{-1})}|\sum_{p\in \Pc}e(p\cdot x)|^{2k}dx.$$
It is rather immediate that
$$\E_{2,\delta}(\Pc)\lesssim |\Pc|^3,$$
\begin{equation}
\label{e20}
\E_{3,\delta}(\Pc)\lesssim |\Pc|^5.
\end{equation}
If we adopt the integral formulation, it is also immediate that
$$\E_{3,\delta}(\Pc)\lesssim |\Pc|^2\E_{2,\delta}(\Pc).$$
These estimates are sharp e.g. when the points lie in a straight line, and are $\delta$-close to an arithmetic progression. However, better estimates are known if the set $\Pc$ is regular, as follows.
\begin{de}\label{d1}
	Given $0<\delta,s<1$,
a $\delta$-separated  set of points $\Pc$ in $[-1,1]$ is $(\delta,s)$-regular if

(i) for each interval $I\subset [0,1]$ with length at least $\delta$ we have $|\Pc\cap I|\lesssim (\frac{|I|}{\delta})^s$

(ii) if in addition the interval is centered at some $p\in \Pc$, then we have   $|\Pc\cap I|\sim (\frac{|I|}{\delta})^s$.	
\end{de}

We have the following discrete reformulation of Theorem \ref{tlasst}.
\begin{te}
\label{t101}
Assume $\Pc$ is $(\delta,s)$-regular.
There is $\eta>0$ depending on $s$ such that
$$\E_{2,\delta}(\Pc)\lesssim |\Pc|^{3-\eta},$$
and thus
$$\E_{3,\delta}(\Pc)\lesssim |\Pc|^{5-\eta}.$$
\end{te}

\medskip

\subsection{Bush estimates}

A bush is a collection of  wave packets $W_T$, satisfying two properties. The spectrum of each $W_T$ needs to be inside an active $\theta$. Also, the underlying tubes $T$ intersect a fixed square $q$ with side length  $R^{1/2}$ (we will say that $\Bc$ is centered at $q$). It is easy to see that we may have at most $O(1)$ tubes for each active $\theta$.

\begin{pr}
\label{p5}	
For each bush $\Bc$ centered at $q$ and for all coefficients $a_T\in\C$ we have
$$\|\sum_{T\in \Bc}a_TW_T\|_{L^6(q)}^6\lesssim_\epsilon R^{\epsilon-7/2}D^3P^2\|a_T\|_{l^\infty}^6.$$
\end{pr}
\begin{proof}
The function $\sum_{T\in \Bc}a_TW_T$ has spectrum inside the $R^{-1}$-neighborhood of $\Gamma$, thus also inside the larger $R^{-1/2}$-neighborhood. We may apply the local inequality in Theorem \ref{t2}
$$\|\sum_{T\in \Bc}a_TW_T\|_{L^6(q)}\lesssim_\epsilon R^\epsilon (\sum_{\tau\text{ active}}\|\sum_{T\in \Bc\atop{T\sim\tau}}a_TW_T\|_{L^6(w_q)}^2)^{1/2}.$$
Recall that each $\tau$ is a region with diameter $\sim R^{-1/4}$. We write $T\sim \tau$ if $T$ is associated with some $\theta$ that sits inside $\tau$. According to (P3), there are $\sim D/P$ active $\tau$. Thus, H\"older's inequality leads to
$$\|\sum_{T\in \Bc}a_TW_T\|_{L^6(q)}^6\lesssim_\epsilon R^\epsilon (D/P)^3 \max_{\tau\text{ active}}\|\sum_{T\in \Bc\atop{T\sim\tau}}a_TW_T\|_{L^6(w_q)}^6.$$
We may assume each active $\theta$ inside $\tau$ contributes only one $T$, and we write $T\sim\theta$. Pick any point $p_\theta\in\theta$ and call $\Pc$ the collection of all $p_\theta$. The spectrum of $W_T$ is inside $B(p_\theta,O(R^{-1/2}))$. We may assume $w_q$ has spectrum inside $B(0,R^{-1/2})$. Note that if $T_i\sim \theta_i$ then the function $W_{T_1}W_{T_2}W_{T_3}\overline{W_{T_4}}\;\overline{W_{T_5}}\;\overline{W_{T_6}}$ has spectrum inside $B(p_{\theta_1}+p_{\theta_2}+p_{\theta_3}-p_{\theta_4}-p_{\theta_5}-p_{\theta_6},O(R^{-1/2}))$. We find that
\begin{align*}\|\sum_{T\in \Bc\atop{T\sim\tau}}a_TW_T\|_{L^6(w_q)}^6&\le \|a_T\|_{l^\infty}^6\sum_{T_1,\ldots,T_6}|\int \Fc(W_{T_1}W_{T_2}W_{T_3}\overline{W_{T_4}}\;\overline{W_{T_5}}\;\overline{W_{T_6}})\Fc(w_q) |\\&\lesssim \|a_T\|_{l^\infty}^6\E_{3,R^{-1/2}}(\Pc)\frac{|q|}{|T_1|^{1/2}\cdot\ldots\cdot|T_6|^{1/2}}\\&\lesssim R^{-7/2}\|a_T\|_{l^\infty}^6P^5.
\end{align*}
\end{proof}

\begin{re}
\label{r101}
Recall that $\tau$ is essentially a rectangle. Because of this, there is no curvature present in the estimate for $\|\sum_{T\in \Bc\atop{T\sim\tau}}a_TW_T\|_{L^6(w_q)}^6$. At this scale, the points $p_\theta$ can be thought of as lying on a line. In the absence of additional structure for these $\theta$, the best exponent we may use for the 3-energy is 5, cf. \eqref{e20}. However, if the measure $\mu$ were AD-regular, we could use the better energy estimate in Theorem \ref{t101}. Such an estimate is enough to produce the gain $R^{-\beta}$ in \eqref{e4}.
\end{re}

\subsection{Incidence estimates and Furstenberg sets}
\label{s6}

Consider the family of rectangles $\Tc=\cup_\theta \Tc_{\lambda,\theta}$. They point in at most $D$ directions, and recall that $D\lesssim R^{\frac{s}2+\alpha}$. Recall also that $|\Tc|\lesssim R^{\frac12+4\alpha}$. We are interested in finding  a good upper bound for the number of (pairwise disjoint) squares $q$ with side length $\sqrt{R}$ that intersect close to the maximum possible number of tubes.
More precisely, we have defined the family
$$\Qc_{heavy}=\{q:\;q\text{ intersects at least }R^{\frac{s}{2}-\alpha}\;\;\text{tubes }T\in\Tc\}.$$
Consider the number of incidences
$$\Ic(\Tc,\Qc_{heavy})=|\{(T,q)\in\Tc\times \Qc_{heavy}:\;T\text{ intersects }q\}|.$$
Let us first point out the following easy estimate, in the spirit of the bilinear Kakeya inequality.
\begin{te}
We have
\begin{equation}
\label{e30}
|\Qc_{heavy}|\lesssim R^{1-s}R^{10\alpha+\frac{2\alpha}{s}}.
\end{equation}
\end{te}
\begin{proof}
We perform a double counting argument. On the one hand we have
$$|\Qc_{heavy}|R^{\frac{s}{2}-\alpha}\lesssim \Ic(\Tc,\Qc_{heavy}).$$
On the other hand, let $\Qc(T_1)=\{q\in \Qc_{heavy}:\;q\text{ intersects }T_1\}$. We claim that for each $T_1\in\Tc$
\begin{equation}
\label{e31}
\Qc(T_1)\lesssim \frac{|\Tc|}{R^{\frac{s}{2}-\alpha-\frac{2\alpha}{s}}}.\end{equation} Let us postpone the proof of this claim to a few lines below, and record the immediate consequence
$$\Ic(\Tc,\Qc_{heavy})\lesssim \frac{|\Tc|^2}{R^{\frac{s}{2}-\alpha-\frac{2\alpha}{s}}}\lesssim R^{1-\frac{s}{2}+9\alpha+\frac{2\alpha}{s}}.$$
The two incidence estimates combine to prove the desired bound for $\Qc_{heavy}$.
\smallskip

Now back to the claim. Fix $T_1\in\Tc$.
Pick any  $q\in \Qc(T_1)$. There must be at least $\sim R^{\frac{s}{2}-\alpha}$ tubes $T$ intersecting $q$, call them $\Tc(q)$, in such a  way that the angle between $T$ and $T_1$ is $\gtrsim R^{-\frac{2\alpha}{s}}$. This is because the number of active $\theta$ (so $\mu(\theta)\sim MR^{-s}$) supported by an arc $I$ of length $\ll R^{-\frac{2\alpha}{s}}$ is
$$\lesssim \frac{\mu(I)}{MR^{-s}}\ll \frac{R^{-2\alpha}}{MR^{-s}}\lesssim R^{\frac{s}{2}-\alpha}.$$
Simple geometry shows that a given $T$ can only be in $\Tc(q_1)\cap \Tc(q_2)$ if $\dist(q_1,q_2)\lesssim R^{\frac12+\frac{2\alpha}{s}}$. In other words, each $T\in\Tc$ may be in at most $R^{\frac{2\alpha}{s}}$ sets $\Tc(q)$.

With this in mind, another double counting proves the claim
$$|\Qc(T_1)|R^{\frac{s}{2}-\alpha}\lesssim \Ic(\cup_{q\in \Qc(T_1)}\Tc(q),\Qc(T_1))\lesssim |\Tc|R^{\frac{2\alpha}{s}}.$$

\end{proof}

To improve the easy estimate, we need to invoke the theory of Furstenberg sets.  Definition \ref{d1} admits the following abstract formulation.
\begin{de}
Let $(X,d)$ be a metric space, let $s,\delta>0$. A set $\Pc\subset X$ is called a $(\delta,s,C)$-set if for each $x\in X$ and $r\ge \delta$ we have
$$|\Pc\cap B(x,r)|\le C(\frac{r}{\delta})^s.$$
\end{de}
We allow $C$ to depend on $\delta$.

We define the distance between two lines in the plane as
$$d(l_1,l_2)=\|\pi_1-\pi_2\|+|a_1-a_2|,$$
where $\pi_i$ is the orthogonal projection onto the linear subspace $\Lc_i$ in the direction of $l_i$, and $a_i=l_i\cap \Lc_i^\perp$. The space $\Ac(2,1)$ of all lines becomes a metric space when equipped with this distance. It is easy to see that a collection of lines  $\Lc$ intersecting the unit square is a $(\delta,s,C)$-set
if each rectangle with dimensions $1$ and $r\in(\delta,1]$  intersects $\lesssim C(\frac{r}{\delta})^s$ lines in $\Lc$, at full length (by that we mean that the intersection is a line segment of length $\sim 1$).
\smallskip

Let $\Lc$ be the collection of central axes (i.e. the lines bisecting the tubes in the long direction) of the rescaled rectangles $R^{-1}T$ with  $T\in \Tc$.
\begin{lem}
\label{l4}	
	The collection $\Lc$ is an $(R^{-1/2},1,R^{4\alpha})$-set.
\end{lem}
\begin{proof}
Fix a rectangle $B$ with dimensions $1$ and $r\ge R^{-1/2}$. The  statement follows from the following observations. For each direction, that is for each active $\theta$, the number of lines in $\Lc$ in this direction that intersect $B$ at full length is at most
$$(r\sqrt{R})^{1-s}\frac{MR^{\frac{s-5}{2}}}{\lambda^2}.$$
This is just a reformulation of \eqref{e12}. Using \eqref{e25} and \eqref{e22}, the expression above  is at most $R^{3\alpha}(r\sqrt{R})^{1-s}$.

The directions of the lines are $R^{-1/2}$- separated, as they correspond to distinct $\theta$'s.  The directions of the lines intersecting $B$ at full length have to lie inside an arc $I$ on the circle with measure $\sim r$. Thus, due to the curvature of $\Gamma$, the corresponding $\theta$'s lie inside a ball $B(x,10r)$. Since for each active $\theta$ we have $\mu(\theta)\sim MR^{-s}$
 and since $\mu(B(x,10r))\lesssim r^s$, we conclude that there can only be $\lesssim \frac{(Rr)^s}{M}$ contributing $\theta$'s. Using again \eqref{e22}, this is at most $R^\alpha (r\sqrt{R})^s$. We conclude that there are at most $R^{4\alpha}r\sqrt{R}$ lines intersecting $B$ at full length.

\end{proof}

Let us write $l(\delta)$ for the $\delta$-neighborhood of the line $l$.

\begin{de}
\label{d17}	
Let $0<s\le 1$, $0<t\le 2$, $\delta\in (0,1)$. A $\delta$-separated set $F\subset \R^2$ is called a discretized  $(\delta,s,t,C)$-Furstenberg set if there is a $(\delta,t,C)$-set of lines $\Lc$ with $|\Lc|\ge C^{-1}\delta^{-t}$ such that $F\cap l(\delta)$ contains a $(\delta,s,C)$-set with cardinality $\ge C^{-1}\delta^{-s}$, for each $l\in\Lc$ .
\end{de}

Each $(\delta,s,t,C)$-Furstenberg set is also a $(\delta,s,t,C')$-Furstenberg set if $C'\ge C$.

We will need the case $t=1$ of the following result, see \cite{OS}.

\begin{te}
If $0<s<1$, $s<t\le 2$, and $\epsilon<\epsilon(s,t)$ then every discretized  $(\delta,s,t,\delta^{-\epsilon})$-Furstenberg set has cardinality $\gtrsim \delta^{-2s-\epsilon}$.
\end{te}

Using this result we can now improve on the cardinality of $\Qc_{heavy}$.

\begin{te}
	\label{t5}
Assume $\alpha<\alpha(s)$ is small enough. Then
\begin{equation}
\label{e44}
|\Qc_{heavy}|\lesssim R^{1-s-2\alpha}.
\end{equation}
\end{te}
\begin{proof}
Assume for contradiction that $|\Qc_{heavy}|\ge R^{1-s-2\alpha}$. We will show that the centers of the squares in $\Qc_{heavy}$ form  a (rescaled) Furstenberg set (with $\delta=R^{-1/2}$), whose cardinality will be forced to be large by the previous theorem. So large, that in fact it  will violate the upper bound \eqref{e30}.

First, let us observe that
$$\Ic(\Tc,\Qc_{heavy})\gtrsim |\Qc_{heavy}|R^{\frac{s}2-\alpha}\ge R^{1-\frac{s}{2}-3\alpha}.$$

Partition $\Tc=\Tc_{light}\cup\Tc_{heavy}$ where $T\in\Tc_{light}$ if $T$ intersects at most $R^{\frac{1-s}{2}-10\alpha}$ squares in $\Qc_{heavy}$.
The axes of the tubes in $\Tc_{heavy}$ will serve as the (rescaled) lines in the definition of the Furstenberg set. We need to prove that there are many such lines.

To achieve this, we first observe that
$$\Ic(\Tc_{light},\Qc_{heavy})\lesssim   |\Tc| R^{\frac{1-s}2-10\alpha}\lesssim R^{1-\frac{s}{2}-6\alpha}.$$
The two incidence estimates force
$\Ic(\Tc_{heavy},\Qc_{heavy})\gtrsim  R^{1-\frac{s}{2}-3\alpha}.$ We may also write, using \eqref{e31}
$$\Ic(\Tc_{heavy},\Qc_{heavy})\lesssim |\Tc_{heavy}|R^{\frac{1-s}{2}+5\alpha+\frac{2\alpha}{s}}.$$
Putting things together, we find
$$|\Tc_{heavy}|\gtrsim R^{\frac12-8\alpha-\frac{2\alpha}{s}}.$$
This lower bound together with Lemma \ref{l4} implies that the central axes of the rescaled tubes $R^{-1}T$, $T\in\Tc_{heavy}$, call them $\Lc_{heavy}$, form a $(R^{-1/2},1,R^{100\alpha+\frac{2\alpha}{s}})$-set.

Take $F$ to consist of the rescaled  centers $R^{-1}c(q)$ of the squares  $q\in \Qc_{heavy}$. It is not hard to see that $F$ is a discretized $(R^{-1/2},1-s, 1,R^{100\alpha+\frac{2\alpha}{s}})$ - Furstenberg set, relative to $\Lc_{heavy}$. We picked the large exponent $100\alpha+\frac{2\alpha}{s}$ for extra safety, and of course, anything larger would also work. The fact that $F\cap l(R^{-1/2})$ is a $(R^{-1/2},1-s,R^{100\alpha+\frac{2\alpha}{s}})$-set for each $l\in\Lc_{heavy}$ follows by an argument similar to the earlier ones, so we omit it. The needed lower bound for $|F\cap l(R^{-1/2})|$ follows from the definition of $\Tc_{heavy}$.
\smallskip

It follows from Theorem \ref{t5} that, if $\alpha$ is chosen small enough so that $100\alpha+\frac{2\alpha}{s}<\epsilon(1-s,1)$, then
$$|\Qc_{heavy}|=|F|\gtrsim R^{1-s}R^{100\alpha+\frac{2\alpha}{s}}.$$
This contradicts \eqref{e30}.

\end{proof}

\subsection{Improved estimates for the heavy squares}

We prove the remaining $L^6$ estimate restricted to the heavy squares, matching the one for the light squares.
\begin{te}
$$\|\sum_{\theta\text{ active}}\sum_{T\in\Tc_{\lambda,\theta}}\langle F_\theta,W_T\rangle W_T\|_{L^6(\cup_{q\in\Qc_{heavy}}q)}^6\lesssim_\epsilon R^{\epsilon-2\alpha}R^{4s-10}.$$
\end{te}
\begin{proof}
Each $q\in\Qc_{heavy}$ determines the bush $\Bc(q)$ consisting of all $T\in \Tc$ that intersect (a slight enlargement of) $q$. Due to rapid decay, we may write for each $x\in q$
$$\sum_{\theta\text{ active}}\sum_{T\in\Tc_{\lambda,\theta}}\langle F_\theta,W_T\rangle W_T(x)=\sum_{T\in\Bc(q)}\langle F,W_T\rangle W_T(x)+O(R^{-100}).$$
Thus, it suffices to estimate
$$\sum_{q\in\Qc_{heavy}}\|\sum_{T\in\Bc(q)}\langle F,W_T\rangle W_T\|_{L^6(q)}^6\;.$$
Using Proposition \ref{p5}, followed by \eqref{e15} and \eqref{e44}, this is dominated by
$$|\Qc_{heavy}|\lambda^6 R^{\epsilon-7/2}D^3P^2\lesssim  R^{\epsilon-2\alpha}R^{-10-s}M^6D^3P^2.$$
The theorem now follows once we notice that $M^6D^3P^2=(MD)^3(MP)^2M\lesssim R^{5s}$, by invoking \eqref{e5}, \eqref{e6} and \eqref{e7}.

\end{proof}
Combining this result with Remark \ref{r7} and  Theorem \ref{t9} delivers the proof of \eqref{e3}, and thus of Theorem \ref{t1}.


\begin{thebibliography}{99}	
\bibitem{BD} J. Bourgain and C. Demeter, {\em The proof of the $l^2$ decoupling conjecture} Ann. of Math. (2) 182 (2015), no. 1, 351-389	
\bibitem{CT} L. Cladek and T. Tao {\em Additive energy of regular measures in one and higher dimensions, and the fractal uncertainty principle}, https://arxiv.org/pdf/2012.02747.pdf
\bibitem{book} C. Demeter, {\em Fourier restriction, decoupling, and applications} Cambridge Studies in Advanced Mathematics, 184. Cambridge University Press, Cambridge, 2020. xvi+331 pp.
\bibitem{DZ} S. Dyatlov, J. Zahl, {\em Spectral gaps, additive energy, and a fractal uncertainty principle} Geom. Funct. Anal.
26 (2016), no. 4, 1011-1094
\bibitem{EV} P, Erdos and  B. Volkmann
{\em Additive Gruppen mit vorgegebener Hausdorffscher Dimension. (German)}
J. Reine Angew. Math. 221 (1966), 203-208
\bibitem{KT} N. Katz and T. Tao {\em Some connections between Falconer's distance set conjecture and sets of Furstenburg type},  New York J. Math. 7 (2001), 149-187.
\bibitem{Gu} Guth, Larry; Iosevich, Alex; Ou, Yumeng; Wang, Hong {\em On Falconer's distance set problem in the plane.} Invent. Math. 219 (2020), no. 3, 779-830
\bibitem{O} T. Orponen, {\em Additive properties of fractal sets on the parabola}, arxiv.org/pdf/2205.02770.pdf
\bibitem{OS} T. Orponen and P. Shmerkin, {\em On the Hausdorff dimension of Furstenberg sets and orthogonal
projections in the plane} arXiv:2106.03338
	
\end{thebibliography}
\end{document}